\documentclass{article}
\usepackage{spconf,amsmath,graphicx}
\usepackage[hidelinks]{hyperref}
\usepackage{amsthm,amssymb}
\usepackage{xcolor}
\usepackage{wrapfig}
\usepackage{algpseudocode}
\usepackage{algorithm}
\usepackage{cleveref}
\usepackage{enumerate}
\usepackage{tcolorbox}
\usepackage{float}
\usepackage{comment}
\usepackage{tikz}
\usepackage{atbegshi}

\newcommand\crnotice{%
\AtBeginShipoutFirst{%
\begin{tikzpicture}[remember picture,overlay] \node[anchor=south,yshift=10pt] at (current page.south) {\fbox{\parbox{\dimexpr\textwidth-\fboxsep-\fboxrule\relax}{ 
\footnotesize \textcopyright Copyright 2026 IEEE. Published in \textit{ICASSP 2026 - 2026 IEEE International Conference on Acoustics, Speech and Signal Processing (ICASSP),} scheduled for 3-8 May 2026 in Barcelona, Spain. Personal use of this material is permitted. However, permission to reprint/republish this material for advertising or promotional purposes or for creating new collective works for resale or redistribution to servers or lists, or to reuse any copyrighted component of this work in other works, must be obtained from the IEEE. Contact: Manager, Copyrights and Permissions / IEEE Service Center / 445 Hoes Lane / P.O. Box 1331 / Piscataway, NJ 08855-1331, USA. Telephone: + Intl. 908-562-3966.}}}; 
\end{tikzpicture} }}

\newtheorem{thm}{Theorem}[section]

\newtheorem{assumption}{Assumption}

\newtheorem{lemma}[thm]{Lemma}
\newcommand{\cO}{\mathcal{O}}

\def\WL{\mathsf{WL}}
\def\BL{\mathsf{BL}}
\def\thr{\texttt{thr}}
\def\R{\mathbb{R}}

\title{Quantile Randomized Kaczmarz algorithm with Whitelist Trust Mechanism}

\name{Sofiia Shvaiko$^{1}$, Longxiu Huang$^{2}$, Elizaveta Rebrova$^{1}$\thanks{This work was partially supported by NSF DMS-2309685 (ER and SS) and by startup funding from Michigan State University (LH).}}
\address{$^{1}$Princeton University \quad $^{2}$Michigan State University}

\begin{document}

\maketitle
\crnotice
\begin{abstract}
 Randomized Kaczmarz (RK) is a simple and fast solver for consistent overdetermined systems, but it is known to be fragile under noise.  
We study overdetermined $m\times n$ linear systems with a sparse set of corrupted equations,
$
{\bf A}{\bf x}^\star = {\bf b},
$
where only \(\tilde{\bf b} = {\bf b} + \boldsymbol{\varepsilon}\) is observed with \(\|\boldsymbol{\varepsilon}\|_0 \le \beta m\).  
The recently introduced QuantileRK (QRK) algorithm addresses this issue by testing residuals against a quantile threshold, but computing a per-iteration quantile across many rows is costly.
In this work we  
(i) reanalyze QRK and show that its convergence rate improves monotonically as the corruption fraction \(\beta\) decreases;  
(ii) propose a simple online detector that flags and removes unreliable rows, which reduces the effective \(\beta\) and speeds up convergence; and  
(iii) make the method practical by estimating quantiles from a small random subsample of rows, preserving robustness while lowering the per-iteration cost.
Simulations on imaging and synthetic data demonstrate the efficiency of the proposed method.
\end{abstract}
\begin{keywords}
 Robust solvers,  randomized Kaczmarz method, stochastic iterative methods, corrupted linear solvers, filtering 
\end{keywords}
\section{Introduction}
\label{sec:intro}
The Kaczmarz method is an iterative technique for solving linear systems that was first introduced in 1937 \cite{karczmarz1937angenaherte}. In computed tomography, it is also referred to as the Algebraic Reconstruction Technique (ART) \cite{gordon1975image,herman1993algebraic,natterer2001mathematics}. The method enjoys a broad range of applications, from tomography to digital signal processing. A major modern development is the \emph{randomized} Kaczmarz algorithm (RK) due to Strohmer and Vershynin \cite{strohmer2009randomized}, which under norm-proportional row sampling achieves an exponential (linear) bound on the expected convergence rate for consistent systems. Real data, however, are rarely perfectly consistent: measurement noise, model mismatch, and even gross corruptions lead to ${\bf A}{\bf x}={\bf b}$ being only approximately solvable. In this regime, classical analyses show that RK converges to a noise-dependent accuracy floor \cite{needell2010randomized,ma2015convergence}, while some variants (e.g., \cite{zouzias2013randomized}) of the method converge to the least squares solution.

Beyond purely random sampling, \emph{adaptive} selection strategies that prioritize informative measurements have also been investigated. “Greedy” rules choose, at each step, a row with a large residual, which improves practical performance while admitting provable convergence guarantees \cite{bai2021,haddock2021greed,haddock2019motzkin}. A complementary line of work develops \emph{quantile-based} randomized Kaczmarz (QRK): rows are sampled, but updates are accepted only when the residual falls below a data-driven quantile threshold, yielding robustness to sparse but large corruptions \cite{steinerberger2023quantile,haddock2020quantile}. Subsampled QRK further reduces per-iteration cost by estimating the quantile from a random residual subset while preserving convergence behavior \cite{haddock2023subsampled,cai2025subsample}. Blocked \cite{cheng2022block}, streaming \cite{jeong2025stochastic}, and time-varying perturbation models have also been analyzed \cite{coria2024quantile}.

Inspired by adversarially distributed RK---which earns trust via simple aggregation statistics and accurately identifies adversarial workers \cite{huang2024randomized}---we apply the same ``earned-trust'' idea to improve the quality of the equations being sampled by QRK method by progressively concentrating updates on a cleaner subset. Concretely, we study the corrupted model 
where a $\beta$-fraction of the observed labels can be arbitrarily corrupted,
and introduce \emph{WhiteList QuantileRK} (WL-QRK) scheme.

On a high level, each iteration samples a mini-batch from the currently whitelisted equations, forms a quantile threshold $q$, and accepts updates only when the chosen row’s residual is below this threshold. Rows whose residuals repeatedly exceed a higher blocking threshold accumulate votes; every $S$ steps, the most suspicious rows are moved to a blocklist and removed from active sampling. Importantly, blocked rows are not discarded permanently: if their residuals later fall below the quantile threshold, they can be returned to the whitelist. 

We provide theoretical support for WL-QRK’s effectiveness: Theorem~\ref{thm:QRKrate} formalizes that QRK converges faster as the corruption rate decreases, while Lemma~\ref{lemma} justifies residual-based screening: after a transient phase, persistently large residuals concentrate on corrupted equations, making residual magnitude a reliable elimination signal.  
       The whitelist/blocklist mechanism is lightweight and adaptive: it reuses QRK’s residuals to (i) detect when residual statistics stabilize and (ii) improve performance by removing  equations with persistently large residuals. 

In summary, our contributions include:
\begin{enumerate}[(i)]
\item  \textbf{Refined QRK analysis.} We establish a refined convergence guarantee for QuantileRK (QRK), showing that the rate constant improves \emph{monotonically} as the corruption fraction $\beta$ decreases.  
\item \textbf{Online row screening.} We introduce a lightweight, per-iteration detector that flags and removes unreliable rows, thereby reducing the \emph{effective} corruption level and accelerating convergence.  
\item \textbf{Subsampled quantiles for speed.} We estimate the residual quantile from a small random subset of rows, preserving QRK’s robustness while cutting per-iteration work from $\cO(m)$ to $\cO(t)$ with $t\ll m$.  
\item \textbf{Empirical validation.} On synthetic data and on real-world tomography data contaminated by  substantial fraction of corruptions, the proposed pipeline consistently outperforms RK and QRK baselines in both speed and accuracy, corroborating the theory. 
\end{enumerate}

\section{Notation, preliminaries and setup} \label{sec:background}
\noindent\textbf{Notation.}
For ${\bf A}\in\mathbb{R}^{m\times n}$ we write ${\bf a}_i$ for its $i$-th row and assume without loss of generality
\begin{equation}\label{eq:row-norm}
  \|{\bf a}_i\|_2=1 \quad \text{for all } i\in [m]:=\{1,\dots,m\}
\end{equation}
(after standard row normalization, scaling the corresponding $b_i$).
Given a vector ${\bf r}\in\mathbb{R}^t$ and $q\in(0,1)$, $\textnormal{Quantile}({\bf r},q)$ denotes the (lower) $q$-quantile of $\{r_\ell\}_{\ell=1}^t$.
We use $\cO(\cdot)$ for order notation. 

\noindent\textbf{Problem model.}
Let ${\bf A}$ be an $m \times n$ measurement matrix and $ \tilde{\bf b}$ be observed labels,
${\bf x}^\star\in\mathbb{R}^n$ and a sparse corruption vector $\boldsymbol{\varepsilon}\in\mathbb{R}^m$ with $\|\boldsymbol{\varepsilon}\|_0\le \beta m$, $\beta\in(0,1)$, such that
\begin{equation}\label{eq:model}
  {\bf A}{\bf x}^\star={\bf b},\qquad \tilde {\bf b} = {\bf b} + \boldsymbol{\varepsilon}.
\end{equation}
The algorithm only observes $({\bf A},\tilde {\bf b})$; the goal is to recover ${\bf x}^\star$, or its high-accuracy approximation.
Throughout, we call a \emph{residual} at ${\bf x}$ the quantity $r_i(x):=\langle {\bf a}_i,{\bf x}\rangle-\tilde b_i$.

\noindent\textbf{Baseline RK. \cite{strohmer2007randomized}}
With rows normalized as in \eqref{eq:row-norm}, a Kaczmarz update using index $k\in[m]$ is the orthogonal projection
\begin{equation}\label{eq:rk-update}
  {\bf x}^{+}={\bf x} -r_k({\bf x})\,{\bf a}_k .
\end{equation}
Classical RK samples $k$ uniformly  at each step (or, proportionally to $\|{\bf a}_k\|_2^2$ if the system is not row-normalized). On inconsistent systems, RK converges to a horizon determined by the noise intensity and oscillates in that region.

\noindent\textbf{QuantileRK (QRK). \cite{haddock2019randomized}}
QRK introduces an acceptance test: at iterate ${\bf x}$, form a batch
$\{i_\ell\}_{\ell=1}^t\subseteq[m]$, compute residuals $|r_{i_\ell}(x)|$, set a threshold
\[
  \tau_q(x):=\textnormal{Quantile}(\big(\{|r_{i_\ell}(x)|\}_{\ell=1}^t,q\big),
\]
and perform the update \eqref{eq:rk-update} only with an index $k$ taken from the \emph{admissible set}
\begin{equation}\label{eq:admissible}
  R_q(x):=\Big\{\,i_\ell:\ |r_{i_\ell}(x)|\le \tau_q(x)\,\Big\}.
\end{equation}
Intuitively, rows whose residuals lie above the $q$-quantile are likely corrupted and are skipped.
Using $t\ll m$ makes the per-iteration cost $\cO(n)+\cO(t)$ (selection in linear time), rather than $\cO(m)$. The full procedure is summarized in Algorithm~\ref{alg:QRK}.

\vspace{-5pt}

 \begin{algorithm}
\caption{Quantile Randomized Kaczmarz (QRK)}

\label{alg:QRK}
\begin{algorithmic}[1]
\Require Matrix ${\bf A}\in\mathbb{R}^{m\times n}$, vector ${\bf b}\in\mathbb{R}^m$, quantile $q\in(0,1)$, iterations $N$, batch size $t$, initial guess ${\bf x}_0$ (default $0$); index pool $\mathcal{I}\subseteq[m]$ (default $\mathcal{I}=[m]$)
\Ensure Approximate solution ${\bf x}_N$
\State ${\bf x} \gets {\bf x}_0$
\For{$j=1,\dots,N$}
  \State  {Sample batch:} draw $i_1,\dots,i_t \sim \mathrm{Uniform}(\mathcal{I})$
  \State  {Residuals on batch:} $r_\ell \gets \langle {\bf a}_{i_\ell},{\bf x}\rangle - b_{i_\ell}$, for $\ell\in[t]$
  \State  {Quantile threshold:} $\tau \gets \mathrm{Quantile}\big(\{|r_\ell|\}_{\ell=1}^t,\, q\big)$
  \State  {Candidate set:} $R \gets \{\, i_\ell : |r_\ell| \le \tau \,\}$
  \State  {Pick an index:} $k \sim \mathrm{Uniform}(R)$ 
  \State  {Kaczmarz step:} ${\bf x} \gets {\bf x} - \big(\langle {\bf a}_k,{\bf x}\rangle - b_k\big)\, {\bf a}_k$
\EndFor
\State \Return ${\bf x} = {\bf x}_N$
\end{algorithmic}
\end{algorithm}

 It has been shown (\cite[Theorem~1.1]{haddock2020quantile}) that QRK converges to ${\bf x}^*$ at the exponential expected convergence rate. We note that neither standard least squares solvers nor standard RK algorithm produce a meaningful solution close to ${\bf x}^*$ in the setting with large corruptions in $\boldsymbol{\varepsilon}$.

\section{Algorithm development}
\label{sec:alg}

We now turn to the theoretical foundations of our approach.  

\begin{assumption}
\label{incoherent}
All the rows $\va_i$ are independent and $\sqrt{n}\va_i$ are mean zero isotropic with uniformly bounded subgaussian norm $\|\sqrt{n}\va_i\|_{\psi_2} \leq K.$ Each entry $a_{ij}$ of $\mA$ has probability density function $\phi_{ij}$ which satisfies $\phi_{ij}(t) \leq D \sqrt{n}$ for all $t\in \R.$ 
\end{assumption}

\begin{thm}\label{thm:QRKrate}
Under the Assumption \ref{incoherent}, with $q=1 - \beta - \alpha$, $\alpha, \beta \in (0, 1)$, \Cref{alg:QRK}  on the full residual has a (detailed) convergence rate 
\begin{align}\label{eq:main}\mathbb{E}&\|{\bf x}_j - {\bf x}^\star\|_2^2 \le \bigg(1 - \frac{C_D (1 - 3\beta - \alpha)^3}{n (1 - 2\beta - \alpha)} \bigg)^j ||{\bf x}_0 - {\bf x}^\star||_2^2 \nonumber\\
&\text{ if } \; \frac{m}{n} > \frac{C_1}{1 - 3\beta - \alpha}\log\frac{DK}{1-3\beta - \alpha}; \nonumber\\ &\quad \frac{4C_K^2\sqrt{\beta}}{\alpha} + \frac{2C_K^2 \beta}{\alpha^2} < \frac{C_D (1 - 3\beta - \alpha)^4}{(1 - 2\beta - \alpha)}.\nonumber
\end{align}
\end{thm}

The proof follows similarly to \cite[Theorem~1.1]{haddock2020quantile}). The key implication of this result is that reducing the fraction of corrupted  equations $\beta$ yields provable acceleration (see Figure \ref{fig:beta}).
 We note that the conditions are satisfied when the matrix is sufficiently tall and $\beta$ is a small enough constant. 
 
\begin{wrapfigure}{r}{0.5\linewidth}
  \centering
  \vspace{-8pt}
  \includegraphics[width=\linewidth]{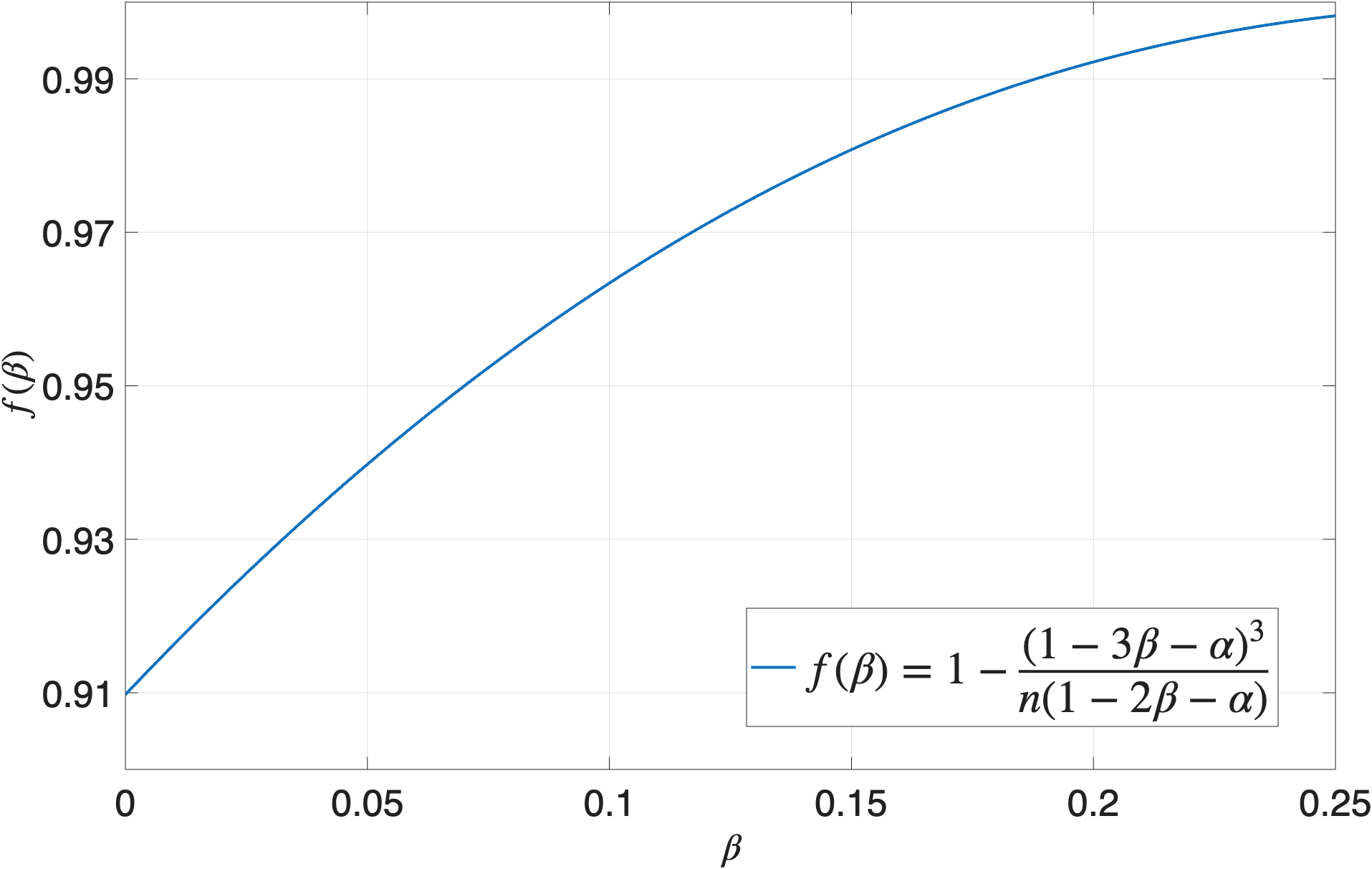}
   
  \caption{The convergence rate in Theorem~\ref{thm:QRKrate} improves as $\beta$ decreases (illustrated with $C_D = 1$, $n = 10$, $\alpha = 0.05$).}
   \vspace{-5pt}
\label{fig:beta}
  \vspace{-15pt}
\end{wrapfigure}
Further, we show that the observed residuals---information collected by the QRK algorithm on the fly---provide a consistent signal for identifying the corrupted equations, after a suitable warm-up period:
\begin{lemma}\label{lemma} Let $s \le \beta m$ (total number of corruptions in the system).  For any $\eta \in (0,1)$ there exists $N_1 = N_1(\eta)$ so that for $j \ge N_1$ the equations corresponding to the top $s$ residuals are all corrupted with probability at least $1 - \eta$.
 \end{lemma}
 We prove the lemma using martingale techniques recently proposed in the context of the iterative solvers in \cite{anderson2025beyond}.
 \begin{proof} Let's set ${\bf e}_j={\bf x}_j -{\bf x}^\star$ and $\varepsilon_{(s)}=$ the $s$-th largest component of $|\boldsymbol{\varepsilon}|$. We have the following two claims.\\
 \textbf{Claim 1:} If $\|{\bf e}_j\| < \varepsilon_{(s)}/2$, then the equations corresponding to the top $s$ residuals are corrupted. 
Claim 1 follows from an argument similar to Lemma~1 from \cite{haddock2019randomized}).
\\
\textbf{Claim 2:} For $\eta \in (0,1),$ there exists $N_1$ so that
\begin{equation}\label{eq:claim-2}
    \mathbb{P} \left\{ \forall j > N_1: \|{\bf e}_j\|^2 < \frac{\varepsilon_{(s)}^2}{4}\right\} \ge 1 - \eta.
\end{equation}
Let's justify Claim 2. 
 From Theorem~\ref{thm:QRKrate}, there exists ${N_1}(\eta)$, such that
$\mathbb{E} \|{\bf e}_j\|^2 \le \frac{\eta \varepsilon_{(s)}^2}{4}.$
For $\ell \ge 0$, we define
 \[
    Z_\ell := \frac{\|{\bf e}_{\ell + N_1}\|^2}{(1 - \rho)^\ell}
\]
with $\rho = \frac{C_D (1 - 3\beta - \alpha)^3}{n (1 - 2\beta - \alpha)}$. Since $ \mathbb{E}_{k-1}\|{\bf e}_\ell\|^2 \le (1 - \rho)\|{\bf e}_{\ell-1}\|^2$ (see \Cref{thm:QRKrate}) $Z_k$
is a nonnegative supermartingale.
So, (e.g.,~\cite[Exercise~4.8.2]{durrett2019probability}) 
\[
    \mathbb{P}\left\{\sup_{\ell \geq 0} Z_\ell \geq \lambda\right\} \leq \frac{\mathbb{E}[Z_0]}{\lambda} = \frac{\|{\bf e}_{N_1}\|^2}{\lambda}.
\]
Set $\lambda=\frac{\varepsilon_{(s)}^2}{4}$ and use  $(1 - \rho)^{j - {N_1}} < 1$ to justify Claim 2.
\end{proof}

\noindent\textbf{WhiteList(WL)-QRK.} To convert transient residual evidence into persistent decisions, we propose the following Algorirm~\ref{alg:WLQRK}.

Here,
we maintain two disjoint sets
$\WL_j,\BL_j\subseteq[m]$ (\emph{whitelist} and \emph{blocklist}) at iteration $j$ with $\WL_j\cup\BL_j=[m]$. Within each mini-batch from $\WL_j$ we compute two quantiles:
(i) the QRK threshold $\tau_q(x)$ used to form the admissible set $R_q(x)$ in \eqref{eq:admissible}, and
(ii) a higher \emph{blocking threshold} $\tau_{\thr}(x)$ (with $\thr>q$). Rows in the mini-batch whose residuals exceed $\tau_{\thr}(x)$ accumulate \emph{block votes}.
A parameter $S$ represents a \emph{blocking cycle} size: every $S$ iterations we (a) move the most suspicious rows with many block votes from $\WL_j$ to $\BL_j$ (lines 19,22),
(b) return rows from $\BL_j$ to $\WL_j$ if their residuals fall below $\tau_q(x)$ (lines 16, 17),
and (c) adapt $q$ to the estimated  corruption on $\WL_j$ (line 23). 

\vspace{-8pt}

\begin{algorithm}[H] \caption{Whitelist QRK (WL-QRK)} \label{alg:WLQRK} \begin{algorithmic}[1] \Require ${\bf A}\in\mathbb{R}^{m\times n}$, ${\bf b}\in\mathbb{R}^m$, corruption fraction (upper bound) $\beta$, confidence gap $\alpha$, warmup iterations $N_{1}$, total iterations $N_{1} + N_{2}$, blocking cycle size $S$, block threshold quantile $\texttt{thr}$, batch size $t$, initial ${\bf x}_0$ (default $0$) \Ensure Approximate solution ${\bf x}$ \State ${\bf x}\gets {\bf x}_0$ \State Initialize lists: $\mathsf{WL}\gets [m]$ (whitelist), $\mathsf{BL}\gets \emptyset$ (blocklist) \State Initialize sampling and blocking counters: $s[i]\gets 0$, $b[i]\gets 0$ for all $i\in[m]$ \State Target quantile: $q \gets 1-\alpha-\beta$ \For{$j=1,\dots,N_{1}+N_{2}$} \State Sample batch: $i_1,\dots,i_t \sim \mathrm{Uniform}(\mathsf{WL})$ \State Residuals: $r_\ell \gets \langle {\bf a}_{i_\ell},{\bf x}\rangle - b_{i_\ell}$ for $\ell\in[t]$ \State q-RK threshold: $\tau_q \gets \mathrm{Quantile}(\{|r_\ell|\}, q)$ \State Block threshold: $\tau_{\texttt{thr}} \gets \mathrm{Quantile}(\{|r_\ell|\}, \texttt{thr})$ \State \emph{Admissible set:} $R \gets \{\, i_\ell : |r_\ell| \le \tau_q \,\}$ \State Pick index $k \sim \mathrm{Uniform}(R)$ \State Update: ${\bf x} \gets {\bf x} - (\langle {\bf a}_k,{\bf x}\rangle - b_k)\,{\bf a}_k$ \State Update counters: $s[i_\ell]\gets s[i_\ell]+1$ for all $\ell\in[t]$ \State Blocked samples: if $|r_\ell|>\tau_{\texttt{thr}}$, then $b[i_\ell]\gets b[i_\ell]+1$ \If{$j>N_{1}$ and $j \bmod S=0$} \State \emph{Return set:} $\mathcal{R} \gets \{\, i\in \mathsf{BL} : |\langle {\bf a}_i,{\bf x}\rangle-b_i|\le \tau_q \,\}$ \State Update lists: $\mathsf{BL}\gets (\mathsf{BL}\setminus \mathcal{R})$, \quad $\mathsf{WL}\gets (\mathsf{WL}\cup \mathcal{R})$ \If {$|\mathsf{BL}|< \beta m$} \State \emph{Discard set:} $\mathcal{D} \gets \{\, i\in \mathsf{WL} : s[i]\ge \tfrac{St}{|\mathsf{WL}|},\; b[i]\ge 0.9\,s[i] \,\}$ \State Reset counters: $s[i]\gets 0,\; b[i]\gets 0$ for all $i$ \EndIf \State Update lists: $\mathsf{BL}\gets (\mathsf{BL} \cup \mathcal{D}$), \quad $\mathsf{WL}\gets (\mathsf{WL} \setminus \mathcal{D}$) \State Update $q$: with $l=|\mathsf{WL}|$, $l_1=|\mathsf{BL}|$, set \[ q \gets 1-\alpha-\tfrac{\beta m - l_1}{l} \] \EndIf \EndFor \State \Return ${\bf x}$ \end{algorithmic} \end{algorithm}

As a result, after every \emph{blocking cycle}, we set aside some equations to not sample from ($\BL_j$) and make a refined estimate for the fraction of corrupted equations remaining in the whitelist, as if all blocked rows were corrupted: $\beta_{new}=\frac{\beta m-\ell_1}{m - \ell_1}$, where $m$ is the total number of rows. To finalize each blocking cycle, quantile parameter $q$ is recalculated based on $\beta_{new}$. While this update is optimistic (even though we add a confidence gap $\alpha$, see line 23), robustness is achieved through  re-introduction: previously blocked rows  return to the reliable set if they appear trustworthy, based on their relatively small residuals. This is theoretically supported by Lemma~\ref{lemma}; 
in practice we rely on the consistency of the top residuals as a heuristic indicator that stabilization has occurred. 

Sampling is restricted to $\WL_j$ for $N_1 + N_2$ iterations total.  This mechanism progressively lowers the effective corruption rate observed by QRK thus speeding it up (as confirmed by  Theorem~\ref{thm:QRKrate}).

 \vspace*{-7pt}

\section{Experiments}
\label{sec:simulation}

 \vspace*{-3pt}

\noindent\textbf{Simulations on synthetic dataset.} For the synthetic experiments, we generate a standardized Gaussian matrix $\mathbf{A}\in\mathbb{R}^{5000\times100}$ and a Gaussian vector $\mathbf{x}^\star \in \mathbb{R}^{100}$. The response is computed as $\mathbf{b} = \mathbf{A}\mathbf{x}^\star$, with $\beta = 0.4$ fraction of the corrupted rows. The indices of the corrupted rows are selected uniformly at random without replacement. To illustrate convergence under different types of corruption, we test three models:

 \vspace*{-5pt}

\begin{itemize}
\item 2-layer: uniform value $u_i \sim \mathrm{Unif}(1,5)$ is added to 1000 rows and a value $u_i \sim \mathrm{Unif}(0.01,0.05)$ to another 1000;
\item 5-layer: a uniform value $u_i \sim \mathrm{Unif}(10^{x-1},10^x)$  is added to 400 rows for each $x \in \{-2,-1,0,1,2\}$;
\item Uniform: a uniform value $u_i \sim \mathrm{Unif}(-5,5)$ is added to the corrupted rows.
\end{itemize}

 \vspace*{-7pt}

\begin{figure}[htbp]
    \centering    \includegraphics[width=0.7\linewidth]{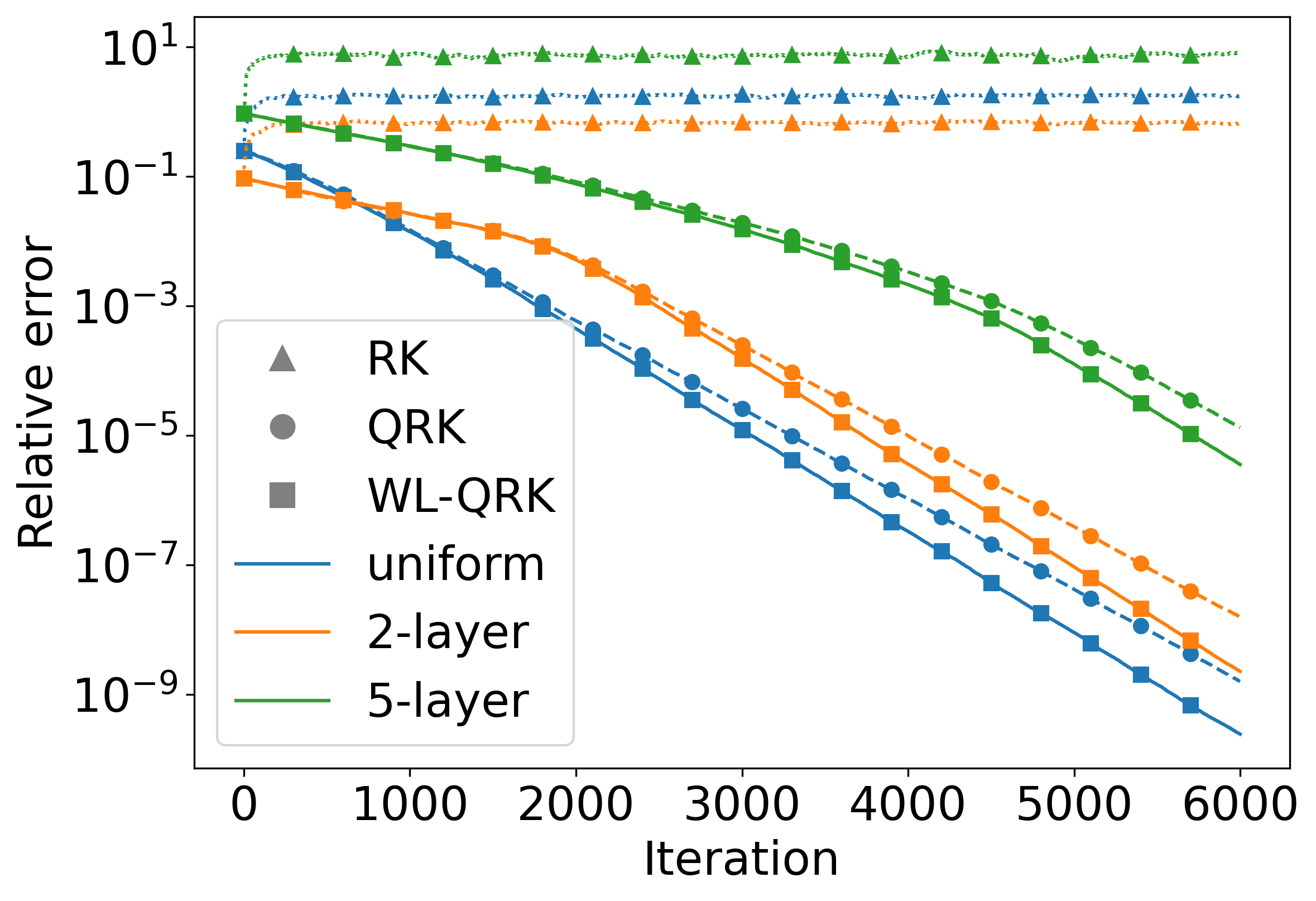}
    \caption{Subsample WL-QRK converges quicker than subsample QRK for various corruption types, RK does not converge.}
    \label{fig:art_ds_conv}
\end{figure}

We tested the algorithms both on the full sample and with subsampling of 2000 rows per iteration. Here, $N_1=100$, $N_2=6000$, and $S=100$. Additionally, the initial guess ${\bf x}_0$ is the least-squares solution of the corrupted system. The results, averaged over 10 runs, are shown in Figure~\ref{fig:art_ds_conv}. 

We also compare taking full samples per iteration with subsampling (\Cref{fig:art_ds_beta_and_fullsamp}). We report the average of the corruption rate in the whitelist $\WL_j$. The figure demonstrates that reducing the sample size from full sampling to 40\% has minimal impact on convergence,
while each iteration of subsampled WL–QRK requires less than half the computations of the full-sample version.

 \begin{figure}[htbp]
    \centering    \includegraphics[width=0.49\linewidth]{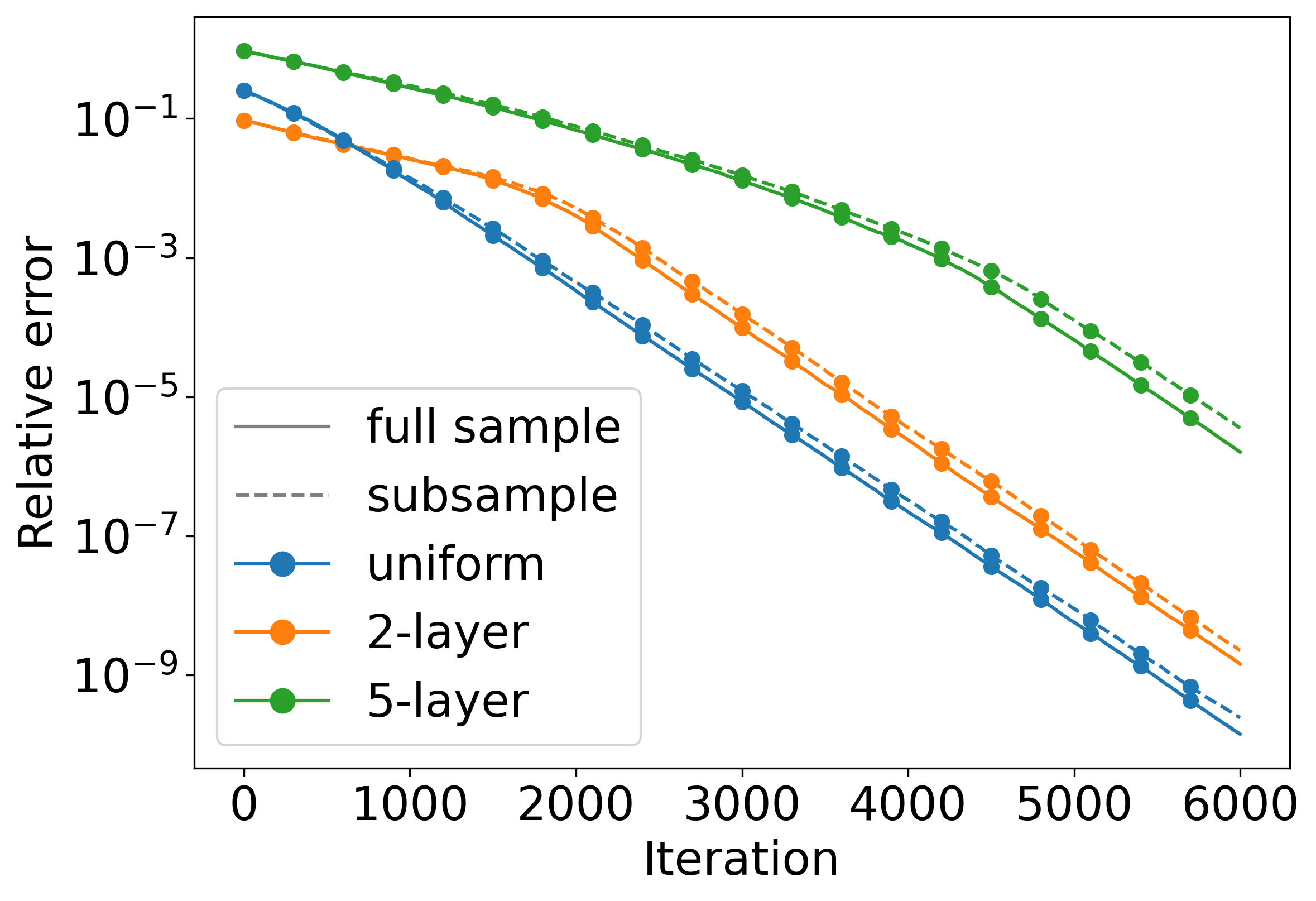}
     \includegraphics[width=0.49\linewidth]{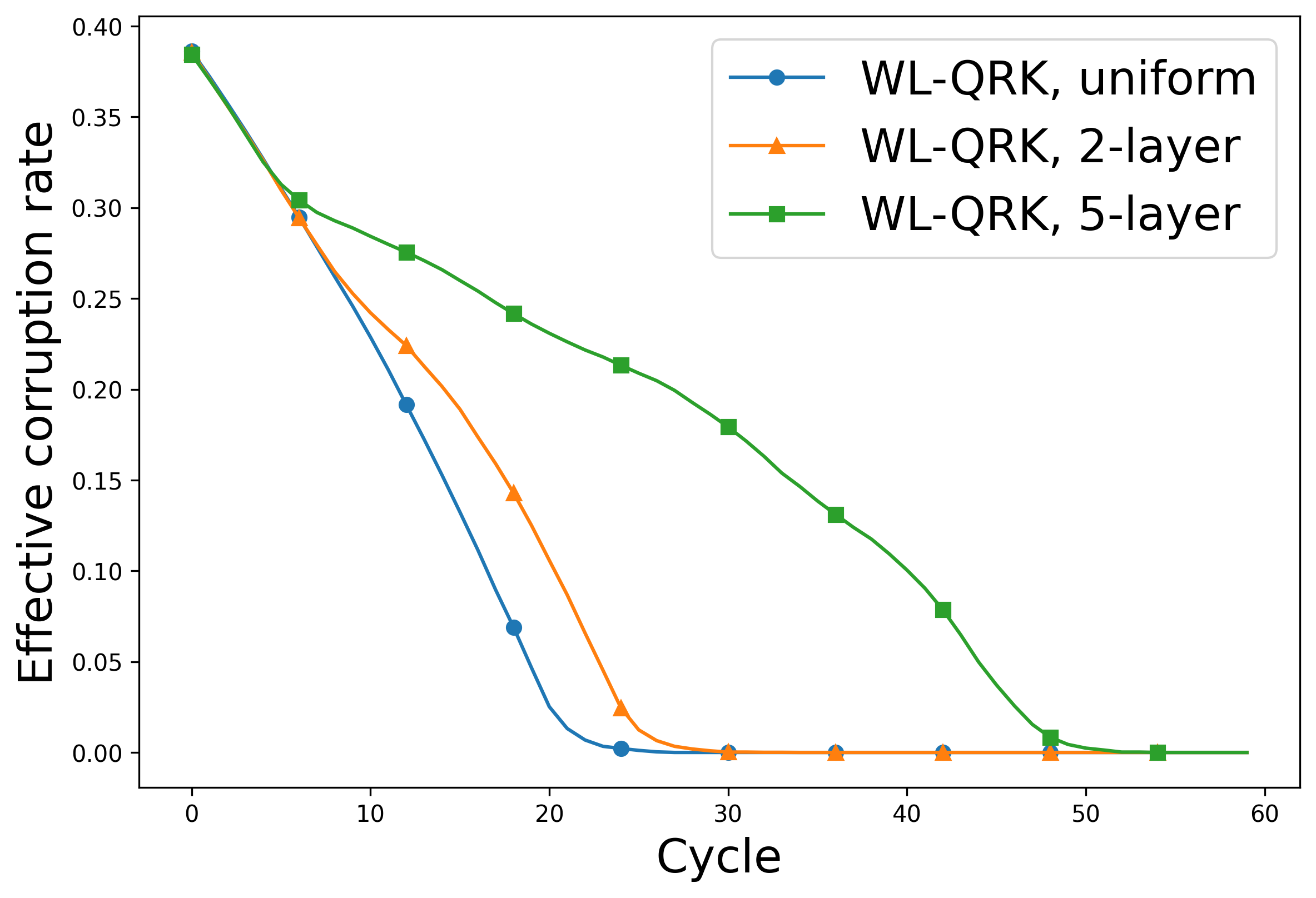} 
    \caption{Subsample WL-QRK performs similarly to the full sample WL-QRK (Left) and the effective corruption rate decreases as the number of blocking cycles grows (Right).}
    \label{fig:art_ds_beta_and_fullsamp}
\end{figure}

 \vspace*{-4pt}
    
\begin{figure}[htbp]
    \centering
    \includegraphics[width=0.49\linewidth]{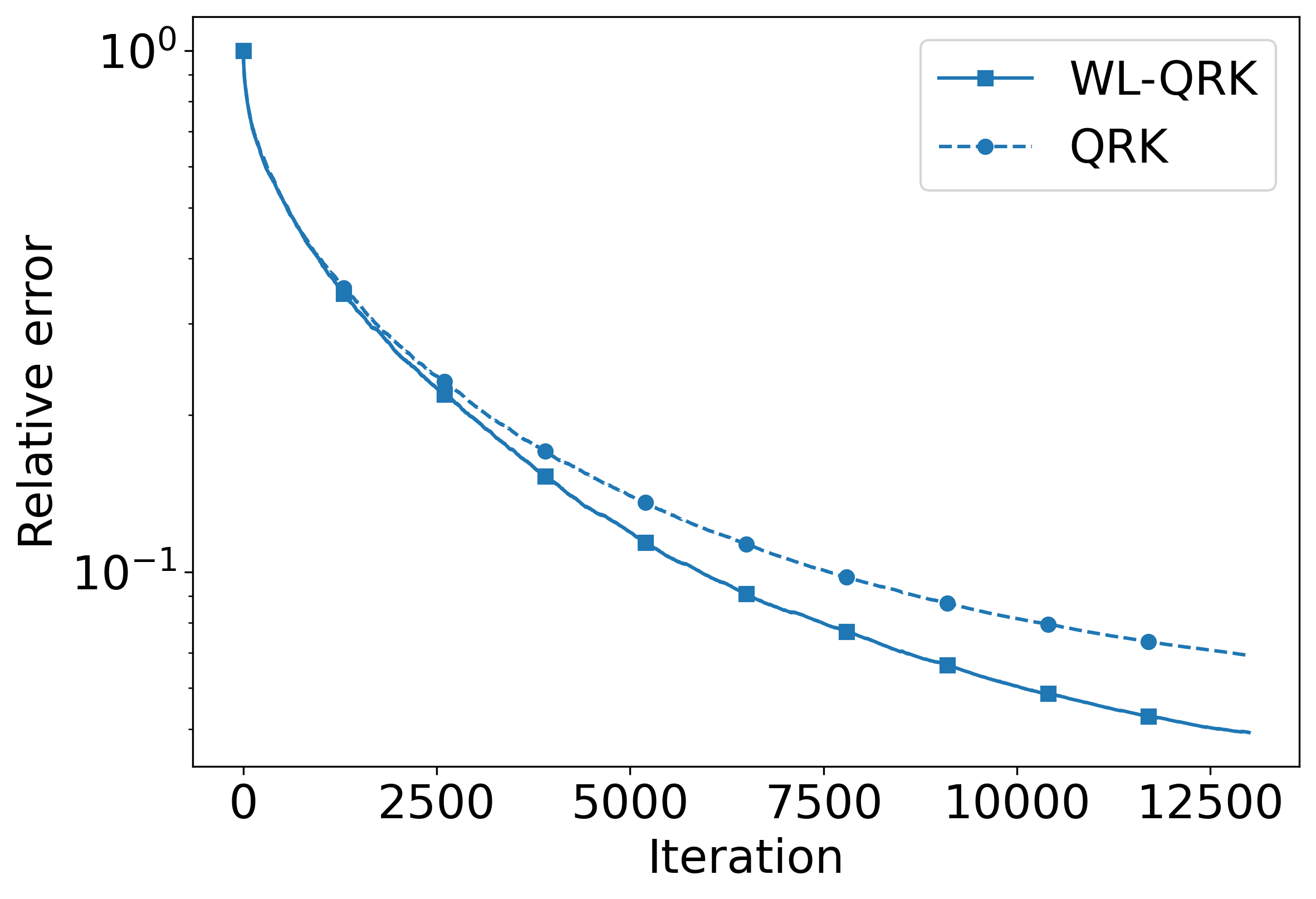} 
\includegraphics[width=0.45\linewidth]{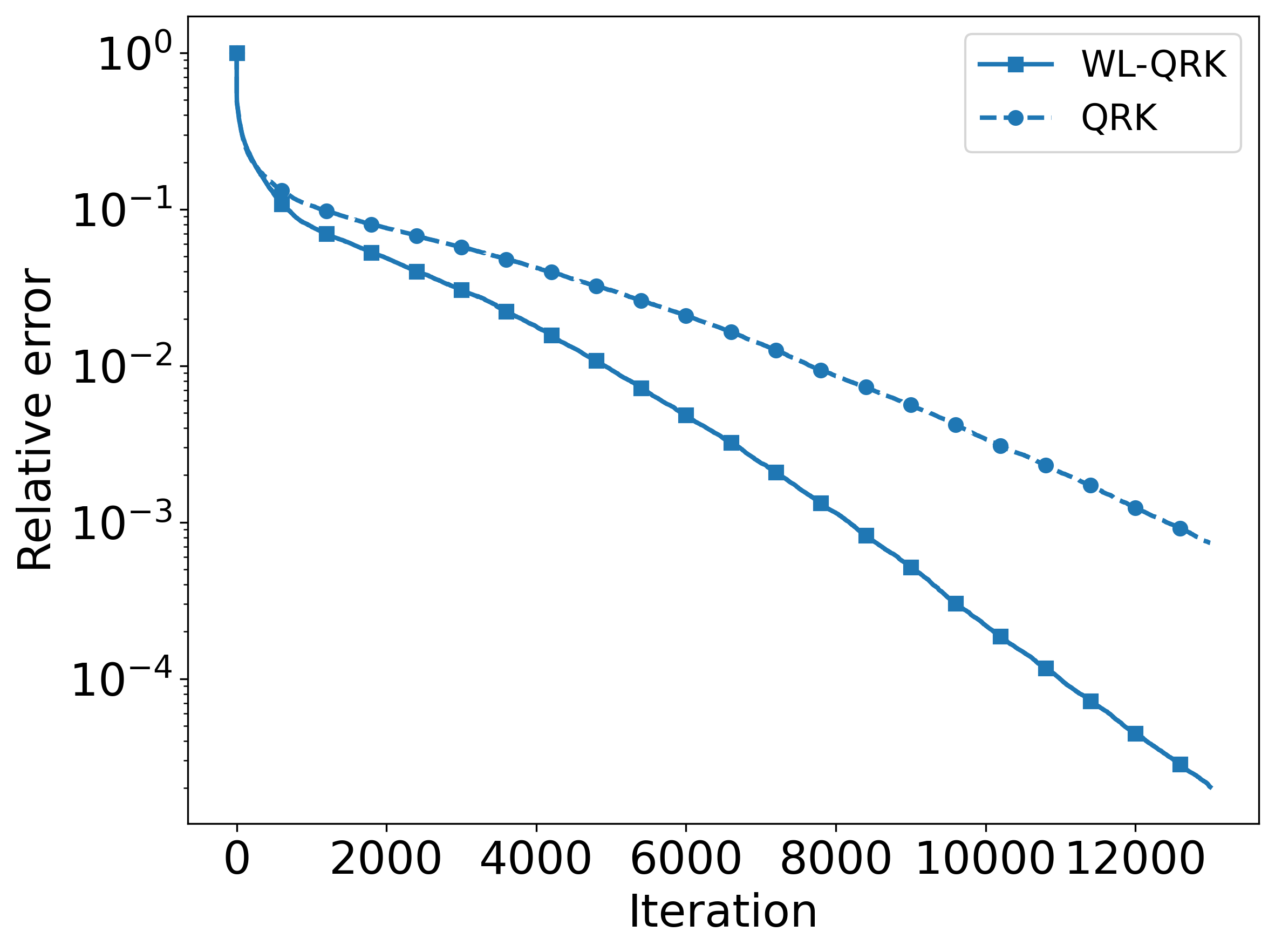} 
    \caption{Full sample WL-QRK v.s. QRK on Tomography system (Left) and WBC system (Right). }
    \label{fig:realdata_conv}
\end{figure}

 \vspace*{-9pt}

\vspace{2mm}
\noindent\textbf{Simulations on real-world dataset.}
We evaluate QRK and its whitelist variant (WL-QRK), initialized with 
${\bf x}_0 = {\bf 0}$, $N_1 = 1000$, $N_2 = 12000$ on two linear systems appearing in medical imaging problems: Wisconsin Breast Cancer (WBC) and tomography problem generated using the MATLAB Regularization Toolbox by P. C. Hansen \cite{hansen2007regularization}. In the WBC experiment, approximately 25\% of the rows
are corrupted, selected uniformly at random, by adding i.i.d. uniform noise perturbations
$u_i \sim \mathrm{Unif}(-20,20)$. In the Tomography experiment, approximately 18\% are corrupted using a 2-layer model: a uniform noise value $u_i \sim \mathrm{Unif}(1,2)$ is added to 100 rows and a uniform noise value $u_i \sim \mathrm{Unif}(40,100)$ is added to 120 rows.

In this simulation, 
each method is run $15$ times 
with independent randomness. As shown in \Cref{fig:realdata_conv}, the 
first $\sim 1$k iterations of the curves are 
nearly identical, since both methods reduce to QRK during the warm-up 
phase. Thereafter, WL-QRK outpaces QRK, demonstrating that the 
whitelist mechanism effectively suppresses the influence of corrupted 
rows---lowering the effective corruption rate---and accelerates 
convergence.

\bibliographystyle{IEEEbib}
\bibliography{refs}

\end{document}